\documentclass[a4paper,UKenglish,cleveref, autoref, thm-restate]{moad}

\newcommand{\R}{\mathbb{R}}

\newcommand{\PP}{\mathcal{P}}
\newcommand{\Q}{\mathcal{Q}}
\newtheorem{assumption}{Assumption} 
\usepackage{cite}
\usepackage{amsmath,amssymb,amsfonts}

\title{ Modulating function-based fast convergent observer for  the Coupled Tanks system} 

\titlerunning{ Modulating function-based fast convergent observer for  the Coupled Tanks system} 

\author{Bahia Hadj Ali }{Laboratoire de Math\'ematiques Pures et Appliqu\'ees. Mouloud Mammeri University of Tizi-Ouzou. {\sc{Algeria}} \and}{bahia.hadjali@ummto.dz}
{https://orcid.org/0009-0005-2705-6500}{(Optional) author-specific funding acknowledgements}

\author{Ania Adil}{Computer, Electrical and Mathematical  Science \& Engineering Division (CEMSE), KAUST, Thuwal 23955-6900.{\sc{Saudi Arabia}}\and}{ania.adil@kaust.edu.sa}{https://orcid.org/0000-0002-2089-0413}{}{}

\author{Fazia Bedouhene }{Laboratoire de Math\'ematiques Pures et Appliqu\'ees. Mouloud Mammeri University of Tizi-Ouzou. {\sc{Algeria}} \and }{fazia.bedouhene@ummto.dz}{https://orcid.org/0000-0002-2664-2445}{}{}

\authorrunning{B. Hadj Ali {\it{et al}} }

\keywords{
Observer design; modulating function; estimation error;  Coupled Tanks system} 

\DateSubmission{September 15, 2024.}
\DateAcceptance{October 3, 2024.}
\begin{document}
\maketitle
\begin{abstract}

In this research, we apply the observer approach introduced by Djennoune et al. \cite{djennoune2019modulating} to estimate water levels in a coupled tanks system. Central to this approach is the use of a remarkable modulating function-based transformation $T_n $, which employs a time/output-dependent coordinate transformation. This transformation converts the original system into a form where the effects of initial conditions are effectively nullified. 
The primary advantage of utilizing the $T_n $ transformation is its ability to achieve instantaneous convergence, ensuring both rapid and accurate state estimation. The observer's finite-time convergence is assured, with the estimation error remaining bounded within a finite period. Numerical simulations further validate the effectiveness of this method for the Coupled Tanks system, demonstrating the robustness of the $T_n $ transformation in practical applications.

\end{abstract}


\section{Introduction}
\label{sec1:intro}

State observers have been extensively studied in recent literature because state variables are crucial in control system theory~\cite{kalman1960new,Luenberger}. 
Traditional design approaches for asymptotic state observers in nonlinear systems utilize linear techniques or coordinate transformations to simplify the system's structure. However, these methods often fail to provide the rapid convergence needed in time-critical applications. In contrast, non-asymptotic observers offer the advantage of driving the estimation error to zero within a prescribed finite time, making them particularly suitable for systems requiring fast responses~\cite{efimov2021finite}.  A novel approach distinct from standard observers has been introduced, leveraging modulating functions to enable non-asymptotic estimation. Originally conceived for parameter identification \cite{PrR:93}, this technique has been subsequently adapted for the combined estimation of parameters and sources, as well as for fault detection \cite{GHAFFOUR, BELKHATIR,aldoghaither2015modulating, fischer2018source, asiri2020} across various linear systems.  An observer for state reconstruction of non-autonomous linear systems is designed in \cite{pin2017deadbeat,pin2020fixed}.  However, extending this method to nonlinear systems has proven challenging, largely due to the intricate mathematical complexities inherent in these systems.

In this work, we present a revised version of the result originally proposed in \cite{djennoune2019modulating}. The goal of this study is to design fast converging observers that not only ensure the estimation error converges to zero but also do so quickly and predictably.  The transformation $T_n $ enables the design of a $\kappa$-fast convergence observer, which is activated after a carefully chosen time delay. This delay is crucial to circumvent potential singularities in the transformation at the initial time $ t = t_0$.  The revision involves modifying certain hypotheses, specifically, adjusting the value of the constant $\kappa$ and updating the condition for the observer's nullity based on the modulating function-based transformation $T_n$. Additionally, we refine elements of the approach, emphasizing a more concise rewriting of the equations to simplify the proofs in \cite{djennoune2019modulating}. 

This paper is structured as follows: Section~\ref{sec_prelim} establishes the theoretical framework necessary for the observer design. Section~\ref{section:3} presents the main results, including the detailed derivation of the $\kappa$-fast convergence observer. Section~\ref{sec_sim} discusses the implementation and application of the proposed observer on water level estimation in coupled tanks system, with a focus on its performance. Finally, Section~\ref{sec_con} concludes the paper. 

         
\section{Preliminaries}\label{sec_prelim}

In this article, we focus on single-input, single-output nonlinear input-affine systems, represented by: 
\begin{equation*}
 \begin{cases}
        \dot{x}(t) &= f(x(t)) + g(x(t))u(t) \\
         y(t) &= h(x(t)),
    \end{cases}
   \end{equation*}
where $x(t) \in \mathbb{R}^n$, $u(t)\in \R$, and $y(t)\in \R$ are the state vector, the input, and  the measured output, respectively. The functions $f: \mathbb{R}^n \rightarrow \mathbb{R}^n$, $g: \mathbb{R}^n \rightarrow \mathbb{R}^n$ and $h: \mathbb{R}^n \rightarrow \mathbb{R}$ are sufficiently smooth real valued vector fields and scalar function, respectively.\\
This system, can be transformed into an observable canonical form using a suitable diffeomorphism as follows:
\begin{equation}\label{equ: z-coordinates}
     \begin{cases}
        \dot{z}(t) &= Az(t) + \Psi(z(t),u(t)) \\
         y(t) &= C(z(t)), \quad t\geq t_0;
    \end{cases}
\end{equation}
where $z(t) \in \R^n$, $u(t)\in \R$, and $y(t)\in \R$ are the state vector, the input, and  the measured output, respectively. The matrices $A$ and $C$ are given under the Brunowsky form, that is:
\begin{equation*}
    A=\begin{bmatrix}
        0 & 1& 0 & \hdots&0\\
        0 &0 & 1 & 0&\vdots\\
        \vdots & \vdots & \vdots & \ddots& \vdots\\
        0 & 0& 0 & \hdots&1\\
         0 & 0& 0 & \hdots&0\\
    \end{bmatrix}; \quad C=\begin{bmatrix}
        1 & 0 & \hdots & 0
    \end{bmatrix}; \text{ and }  \Psi(z(t),u(t))=\begin{pmatrix}
        0\\
        \vdots\\
        0\\
        f(z(t))+g(z(t))u(t)
    \end{pmatrix} 
\end{equation*}


The following assumptions are made:
\begin{assumption}\label{assump}
    \begin{enumerate}
        \item The pair $(A,C)$ is observable;
        \item There exist four positive constants $\gamma_f, \delta_f, \gamma_g, \delta_g$ such that $\forall z_1, z_2 \in \R^n$
    \begin{equation}
            ||f(z_1)-f(z_2)||^2\leq \gamma_f^2||z_1-z_2||^2+ \delta_f^2 \quad \text{ and }   \quad        
            ||g(z_1)-g(z_2)||^2\leq \gamma_g^2||z_1-z_2||^2+ \delta_g^2
    \end{equation}
     \item \label{assump4}
The input $u(t)$ is bounded by a positive constant $M_u$, i.e. $|u(t)| \leq M_u, \forall t \geq t_0$.
    \end{enumerate}
\end{assumption}


\begin{definition}[\cite{djennoune2019modulating}]
Given the nonlinear system \eqref{equ: z-coordinates}, an observer 
$\dot{\hat{z}}(t) = \mathcal{F}(\hat{z}(t), u(t), y(t))$ for \eqref{equ: z-coordinates} is said to be $\kappa$-fast convergent with a prescribed finite-time convergence $t_a$ if there exists $\kappa>0$ such that for any initial conditions $z_0$ and $\hat{z}_0$ with $z_0 \neq \hat{z}_0$, the error $e(t) = z(t) -\hat{z}(t)$ satisfies 
   $ \|e(t)\| \leq \kappa \quad \text{for all } t \geq t_a.$ 
   
If $\kappa = 0$, the observer is termed an exact-fast convergent observer.
\end{definition}
It should be noted that $\kappa$ is independent of the initial conditions.


\begin{definition}[Modulating Function, \cite{djennoune2019modulating,aldoghaither2015modulating}]\label{MF}
Let $m\in\mathbb{N^*}$, and $\mu$ be a function satisfying the following properties:
\begin{align}
   \left\{\begin{array}{llll}
   &\mu\in\mathcal{C}^{m-1}([t_0,\infty[);\\
   &\mu^{(j)}(t_0)= 0,\,\forall j=0,1,\ldots,m-1;\\
   &\mu^{(j)}(t) \neq 0 \quad\text{for }\quad t > t_0, \forall j = 0, 1, \dots, m-1;\\
   &\sup_{t\geq t_0}|\mu^{(j)}(t)| \leq M_j
    \end{array}\right.
\end{align}
   Then, $\mu$ is called the $m^{th}$ order modulating function on $[t_0,\infty[$. 
\end{definition}

     
\section{Design of the modulating function based fast convergent observer}
\label{section:3}
In this section, we present the theoretical extension of the time-output transformation approach by Djennoune et al. \cite{djennoune2019modulating} to the nonlinear system \eqref{equ: z-coordinates}. This extension constructs a modulating function-based observable that modulates initial conditions to zero. Additionally, we review and refine some elements of the approach, focusing on a compact rewriting of the equations, which simplifies the proofs in \cite{djennoune2019modulating}. Furthermore, the assumption of initial conditions for the constructed observer has been removed.

        
\subsection{Magnificent Modulating Function-Based Transformation Applied to the Nonlinear System}

Given a $n$-order modulating function  $\mu$  and a sequence $(\alpha_{ji})$ defined by \begin{equation*}
\left\{ 
\begin{array}{l}
\alpha_{j,0} = 1,\ j=1,\dots,n; \\ 
\alpha_{n,i} = (-1)^{i},\ i=0,\dots,n-1; \\ 
\alpha_{j,i} = \alpha_{j+1,i} - \alpha_{j,i-1},\ j=1,\dots,n-1;\ i=1,\dots,j-1.
\end{array}
\right.
\end{equation*}
Let \begin{equation*}
T_{n}(\mu (t)) = 
\begin{bmatrix}
\alpha_{1,0}\mu (t) & 0 & 0 & \dots  & \dots  & 0 \\ 
\alpha_{2,1}\mu^{(1)}(t) & \alpha_{2,0}\mu (t) & 0 & \dots  & \dots  & 0 \\ 
\alpha_{3,2}\mu^{(2)}(t) & \alpha_{3,1}\mu^{(1)}(t) & \alpha_{3,0}\mu (t) & \dots  & \dots  & 0 \\ 
\vdots  & \vdots  & \vdots  & \ddots  & \vdots  & \vdots  \\ 
\alpha_{n-1,n-2}\mu^{(n-2)}(t) & \alpha_{n-1,n-3}\mu^{(n-3)}(t) & \dots & \dots  & \alpha_{n-1,0}\mu (t) & 0 \\ 
(-1)^{n-1}\mu^{(n-1)}(t) & (-1)^{n-2}\mu^{(n-2)}(t) & \dots  & \dots  & -\mu^{(1)}(t) & \alpha_{n,0}\mu (t)
\end{bmatrix}
\end{equation*}

From these definitions, we can deduce additional properties of the sequence $\alpha_{j,i}$, in particular:
\begin{eqnarray*}
\alpha_{n-1,i} &=& (-1)^{i}(i+1);\ i=0,\dots,n-2 \\
\alpha_{n-2,i} &=& (-1)^{i}\dfrac{(i+1)(i+2)}{2};\ i=0,\dots,n-3
\end{eqnarray*}

The transition from $T_{n}(\mu(t))$ to $T_{n+1}(\mu(t))$ can be easily achieved by the following steps:
\begin{eqnarray*}
\alpha_{j,i} &=& \alpha_{j-1,i}, \quad j=1,\dots,n-1,\ i=1,\dots,j-1, \\
\alpha_{j,j-1} &=& \alpha_{j+1,j-1} - \alpha_{j,j-2}, \quad j=1,\dots,n.
\end{eqnarray*}
This implies that the only elements that need to be calculated in $T_{n+1}$ are those in the first column, while the remaining elements are the same as in $T_n$.

Clearly, $T_{n}(\mu(t))$ is invertible for every $t > t_0$. Moreover, if $\xi(t)$ is the image of the state $z(t)$ under this transformation, i.e., \begin{equation}\label{eq:xi}
    \xi(t) = T_n(\mu(t))z(t).
\end{equation}
The derivative of $\xi(t)$ is given by
\begin{eqnarray*}
 \dot{\xi}(t)& = &T_{n}(\mu^{(1)}(t))z(t) + T_{n}(\mu(t))\dot{z}(t)\\&=&T_{n}(\mu^{(1)}(t))T_{n}^{-1}(\mu(t))\xi(t) + T_{n}(\mu(t))(AT_{n}^{-1}(\mu(t))\xi(t)+\Psi(T_{n}^{-1}(\mu(t))\xi(t),u(t)))  \\&=&\left(T_{n}(\mu^{(1)}(t))+ T_{n}(\mu(t))A\right)T_{n}^{-1}(\mu(t))\xi(t)+T_{n}(\mu(t))\Psi(T_{n}^{-1}(\mu(t))\xi(t),u(t)))\\
 &=&G_n(\mu(t))\xi(t)+H_{n}(\mu(t),\xi(t),u(t)).
\end{eqnarray*}
In addition, we have $ \xi(t_0)=0_{n\times 1}$.
The following lemma provides an explicit expression for the matrices $G_n(\mu(t)) = \left(T_{n}(\mu^{(1)}(t)) + T_{n}(\mu(t))A\right)T_{n}^{-1}(\mu(t))$ and $H_{n}(\mu(t),\xi(t),u(t))=T_{n}(\mu(t))\Psi(T_{n}^{-1}(\mu(t))\xi(t),u(t)))$.
\begin{lemma}
For any \(t > t_0\), the following hold:
    \item \[
   G_n(\mu(t)) = \begin{bmatrix}
    \alpha_{1,0}\mu^{(1)}(t)/\mu(t) & 1 & 0 & \dots    & 0 \\ 
    \alpha_{2,1}\mu^{(2)}(t)/\mu(t) & 0 & 1 & \dots   & 0 \\  
    \vdots  & \vdots  & \vdots  & \ddots& \vdots  \\ 
    \alpha_{n,n-1}\mu^{(n-1)}(t)/\mu(t) & 0 & \dots & \dots  & 1 \\ 
    (-1)^{n}\mu^{(n)}(t)/\mu(t) & 0 & \dots  & \dots  & 0
    \end{bmatrix}; \; H_{n}(\mu(t),\xi(t),u(t)) = \begin{bmatrix}
    0 \\
    0 \\
    \vdots \\
   \tilde{\Psi}(\mu(t),f(t),g(t),u(t))
    \end{bmatrix}
    \] where $   \tilde{\Psi}(\mu(t),f(t),g(t),u(t))=\mu(t) \left(f(T_{n}^{-1}(\mu(t))\xi(t)) + g(T_{n}^{-1}(\mu(t))\xi(t))u(t)\right)$ and the coefficients $\alpha_{j,j-1}$, $j=1,\dots,n$ are those of $T_{n+1}(\mu(t))$.  
\end{lemma}
In the sequel, we use the following notations $\tilde{a}(\xi(t)) =  f(T_{n}^{-1}(\mu(t))\xi(t))$ and $\tilde{b}(\xi(t)) = g(T_{n}^{-1}(\mu(t))\xi(t))$, for $t>t_0$.

        
\subsection{Construction of the modulating function based $\kappa$-fast convergent observer}

The modulating function-based observer proposed in \cite{djennoune2019modulating} is originally designed as a conventional Luenberger-type observer. However, we present it here in a more compact form, incorporating the matrix \( G_n \) for clarity and efficiency. First, observe that  \begin{eqnarray}\label{eq:refor G_n}
    G_n(\mu(t))\xi(t)&=&A\xi(t)+\begin{bmatrix}
    \alpha_{1,0}\mu^{(1)}(t)&
    \alpha_{2,1}\mu^{(2)}(t)  &
    \hdots& 
    (-1)^{n}\mu^{(n)}(t)
    \end{bmatrix}^T \xi_1(t)/\mu(t)\\
    &=&A\xi(t)+B_\alpha (\mu(t)) y(t)
\end{eqnarray}
Thus, under \eqref{eq:refor G_n}, the differential equation satisfied by $\xi$ takes the following form:
\begin{equation}\label{equ:dot-xi}
  \begin{cases}
\dot{\xi}(t)=A\xi(t)+B_\alpha (\mu) y(t)+B_0\mu(t)\left (\tilde{a}(\xi)+\tilde{b}(\xi)u(t)\right) \\  \xi_1(t)=\mu(t)y(t); \;
\xi(t_0)=0.
    \end{cases}  
\end{equation}
where $B_0=\begin{bmatrix}
    0& 0&\dots&1
\end{bmatrix}^T$.
Thus, instead of using the following observer structure,
\begin{equation}\label{equ:dot-hatxi}
\dot{\hat{\xi}}=G_n(\mu(t))\hat{\xi}(t)+H_{n}(\mu(t),\hat{\xi}(t),u(t))+\mu(t)K\left( y(t)-\hat{y}(t) \right),
\end{equation}
we employ a more suitable form that leverages the fact that \(\xi_1\) is known and does not require estimation, namely:
\begin{equation}\label{equ:dot-hatxi}
  \begin{cases}
\dot{\hat \xi}(t)=A\hat \xi(t)+B_\alpha (\mu) y(t)+B_0\mu(t)\left (\tilde{a}(\hat\xi)+\tilde{b}(\hat\xi)u(t)\right) +\mu(t)K\left( y(t)-\hat{y}(t) \right)\\  \hat \xi_1(t)=\mu(t)\hat y(t); \;
\hat \xi_1(t_0)=0.
    \end{cases}  
\end{equation}
where $K^T=\begin{bmatrix}
    k_1 &k_2& \dots &k_n
\end{bmatrix}$ is the observer gain to be determined. Contrarily to \cite{djennoune2019modulating}, here $\hat{\xi}_j(t_0) $, $j=2,\dots,n$ are arbitrarily. 
The estimate \(\hat{z}(t)\) of the system's original state $z(t)$ is derived through the expression \begin{eqnarray}\label{eq:hat z}
    \hat{z}(t) = T_n^{-1}(\mu(t))\hat{\xi}(t) \quad \text{for} \quad t > t_0.
\end{eqnarray}
Since \(T(\mu(t))\) is singular at \(t = t_0\), the observer activation should be delayed to avoid instability. We can determine an activation time \(t_a\) such that, for a fixed \(\epsilon\), \({\rm det} \, T_n(\mu(t)) > \epsilon\) for all \(t \geq t_a\). For \(\mu(t) = (1 - e^{-(t - t_0)})^n\), \({\rm det} \, T_n(\mu(t)) = (1 - e^{-(t - t_0)})^{n^2}\). Given \(\epsilon > 0\) sufficiently small, \(t_a\) is computed as:
\begin{equation}\label{eq:t_a}
  t_a > t_0 -\ln\left(1 - \epsilon^{1/n^2}\right).
\end{equation}
In the following, we present a revised version of the theorem originally proposed in \cite[Theorem~1]{djennoune2019modulating}. In our revision, we have adjusted certain hypotheses, specifically, the value of the constant $\kappa$ has been modified and the condition for the nullity of the observer, based on the transformation $T_n$. 

          
\begin{theorem}
Given the system \eqref{equ: z-coordinates} under Assumption~\ref{assump} and a $n$-order modulating function $\mu$ that transforms \eqref{equ: z-coordinates}  into the observer form \eqref{equ:dot-xi}. Let the 
modulating function based observer \eqref{equ:dot-hatxi}, with gain $K$ chosen such that there exist symmetric positive matrices $P$ and $Q$ satisfy the Lyapunov equation: $(A-KC)^T\PP+\PP(A-KC)=-Q$.
Let  $t_a > t_0$ be such that $T(\mu(t))$ is invertible for $t \geq t_a$. If
\begin{equation}
        \varpi=\frac{\lambda_{min}(\Q)}{\lambda_{max}(\PP)}-M_0\lambda_{max}(\PP) \left( 1 + M_u^2 + \frac{\gamma_f^2+\gamma_g^2}{\lambda_{max}(\PP)} \right) > 0,
\end{equation}
then, there exists a constant $\kappa>0$, with\begin{equation*}
\kappa = \tau_{t_a} \frac{\lambda_{\max}(\PP)}{\lambda_{\min}(\PP)} \|\hat{\xi}(t_0)\|^2 e^{-\varpi (t_a - t_0)} + \tau_{t_a} \frac{M_0(\delta_f^2 + \delta_g^2)}{\varpi \lambda_{\min}(\PP)} \left(1 - e^{-\varpi (t_a - t_0)}\right)>0,
\end{equation*} where $\tau_{t_a}=\max_{t\geq t_a}||T^{-1}(\mu(t))||$,  such that $||e(t)||\leq \kappa, \forall t\geq t_a > t_0$. 
 Thus, the observer \eqref{eq:hat z} is $\kappa$-fast convergent to \eqref{equ: z-coordinates}.
\end{theorem}

          
\begin{proof}
We provide here a brief outline of the proof, starting from equation (40) in the paper \cite[eq:(40)]{djennoune2019modulating}:  
$\dot{V}(\Tilde{e}(t)) \leq -\varpi \lambda_{\max}(\PP) \|\Tilde{e}(t)\|^2 + M_0(\delta_a^2 + \delta_b^2)$, where $\Tilde{e}(t)=\xi(t)- \hat \xi (t)$. 
Applying Rayleigh's Inequality and Gronwall's Inequality to this later yields:
\begin{equation*}
\|\Tilde{e}(t)\|^2 \leq \frac{\lambda_{\max}(\PP)}{\lambda_{\min}(\PP)} \|\hat{\xi}(t_0)\|^2 e^{-\varpi (t - t_0)} + \frac{M_0(\delta_f^2 + \delta_g^2)}{\varpi \lambda_{\min}(\PP)} \left(1 - e^{-\varpi (t - t_0)}\right).
\end{equation*}
for all \( t > t_0 \). For \( t > t_a \), and letting \( \tau_{t_a} = \max_{t \geq t_a} \|T^{-1}(\mu(t))\|^2 \), we obtain:
\begin{eqnarray*}
\|z(t) - \hat{z}(t)\|^2 &\leq & \|T^{-1}(\mu(t))\|^2 \|\Tilde{e}(t)\|^2 \\
&\leq & \tau_{t_a} \frac{\lambda_{\max}(\PP)}{\lambda_{\min}(\PP)} \|\hat{\xi}(t_0)\|^2 e^{-\varpi (t_a - t_0)} + \tau_{t_a} \frac{M_0(\delta_a^2 + \delta_b^2)}{\varpi \lambda_{\min}(\PP)} \left(1 - e^{-\varpi (t_a - t_0)}\right):=\kappa.
\end{eqnarray*}
This constant \(\kappa\), which depends on \(t_a\) but not on the initial conditions \(z(t_0)\) and \(\hat{z}(t_0)\), can be interpreted as a convex combination of the two constants \(\tau_{t_a} \frac{\lambda_{\max}(\PP)}{\lambda_{\min}(\PP)} \|\hat{\xi}(t_0)\|^2\) and \(\tau_{t_a} \frac{M_0(\delta_a^2 + \delta_b^2)}{\varpi \lambda_{\min}(\PP)}\), weighted by \(e^{-\varpi (t_a - t_0)}\). If we assume that \(\hat{\xi}(t_0) = 0\), we recover the result in \cite{djennoune2019modulating}.
\end{proof}

            
\section{Example and Simulation study}\label{sec_sim}
In this example, we apply the theoretical results presented in Section 3 to estimate water levels in a coupled tanks system. This type of system is commonly used in various industrial processes, such as wastewater treatment and food production. Controlling these systems is challenging due to their complex structure and the presence of unmeasurable variables.
The dynamical model of the coupled tanks, expressed in canonical form, is given by
\begin{figure}[H]
  \centering
    \begin{minipage}{0.6\textwidth}
  \begin{equation}\label{eq:sys coupled tanks}
    \begin{cases}
        \dot{z}_1&= z_2\\
        \dot{z}_2&=\varphi(z,u) \\
        y&= z_1
    \end{cases},
\end{equation}
where $$\varphi(z,u)=  - \frac{A_{o1}^2 g}{A_{t1}A_{t2}} \left( 1 + \frac{K_p V_p}{A_{t2}z_2 + A_{o2} \sqrt{2gz_1}} \right) -\frac{A_{o2}g}{A_{t2}}\left( \frac{z_2}{\sqrt{2gz_1}} \right)$$
$z_1$, $z_2$ represent the water levels in tanks 2 and 1, respectively. $V_p$ denotes the pump voltage (the input), while $y$, the water level in tank 2, serves as the output of system (\ref{eq:sys coupled tanks}).
The values of the physical parameter of system can be found in \cite{adil2021coupled}.\\
The modulating function is selected as 
\(\mu(t) = (1-e^{-t})^2\), and the transformation \(T\) takes the following form: \[
T_2(\mu(t))= \begin{bmatrix}
        \mu(t) & 0\\
        -\mu^{(1)}(t) & \mu(t) 
    \end{bmatrix}.
\]
    \end{minipage}
    \hfill
    \begin{minipage}{0.3\textwidth}
       \centering
        \includegraphics[width=\textwidth]{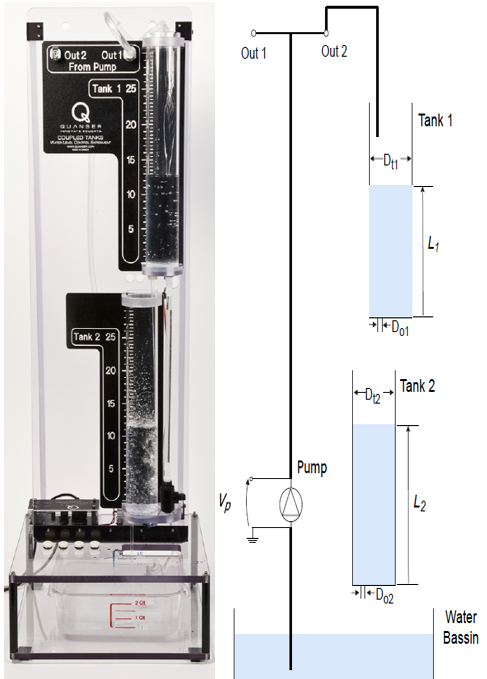}
        \caption{ Coupled Tanks
system \cite{adil2021coupled}. }
    \end{minipage}
\end{figure}
To ensure the invertibility of \(T_2(\mu(t))\), it is sufficient to choose \(t_a > 0.38\), guaranteeing \( \mathrm{det}\,T_2(\mu(t))>0.01\).
The image of system \eqref{eq:sys coupled tanks} via the transformation $T$ takes the following form:
\[
    \begin{cases}
       \dot{\xi}_1(t) = \xi_2(t) + 2\mu^{(1)}(t)y(t),\\
       \dot{\xi}_2(t) = \mu(t) \varphi(\xi(t)) + \mu^{(2)}(t)y(t); \;     \xi_1(t)=\mu(t)  y(t)
    \end{cases}
\]
The modulating function based fast convergent observers $\hat{\xi}$ and $\hat{z}$ are given respectively by:
\begin{equation*}
    \begin{cases}
        \dot{\hat{\xi}}_1(t)= \hat{\xi}_2(t) + 2\mu^{(1)}(t)y(t) + k_1\mu(t)(y(t)-\hat{y}(t)),\\
        \dot{\hat{\xi}}_2= -\mu^{(2)}(t)y(t) + k_2\mu(t)(y(t)-\hat{y}(t)), \;     \hat \xi_1(t)=\mu(t)  \hat y(t) ,
    \end{cases} \quad \hat{z}(t)=\begin{cases}
        0,& t<t_a\\
        T_2^{-1}(\mu(t))\hat{\xi}(t),& t\geq t_a
    \end{cases}
\end{equation*}
where the nonlinear function $\varphi$ is assumed to be unknown.

\subsection{Simulation study}
From the simulations, we observe that the observer converges to the original system immediately after \( t_a = 0.38013 \, \text{s} \) where $\epsilon=0.01$, without exhibiting any transient peaks. These simulations are performed using the initial conditions \(\hat{\xi}_0 = [0; 4]\), where the first component is zero by definition (see \eqref{equ:dot-xi}), \(z_0 = [4; 4]\), and the observer gain $K =[30\hspace{0.3cm} 200]^T $.

\begin{figure}[H]
  \centering
    \begin{minipage}[t]{0.45\textwidth}
        \centering
        \includegraphics[width=\textwidth]{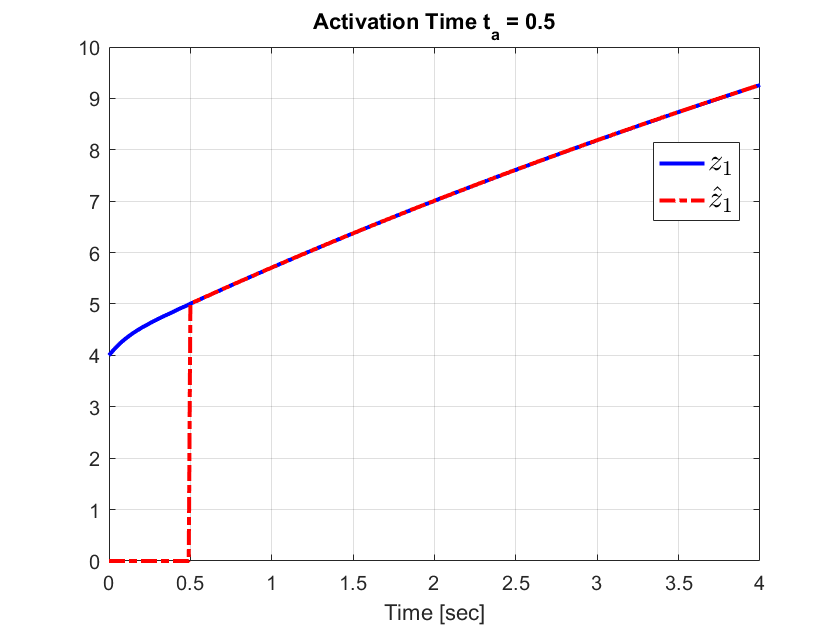}
        \caption{Behaviour of $z_1$ and its estimate $\hat{z}_1$}
    \end{minipage}
    \hfill
    \begin{minipage}[t]{0.45\textwidth}
       \centering
        \includegraphics[width=\textwidth]{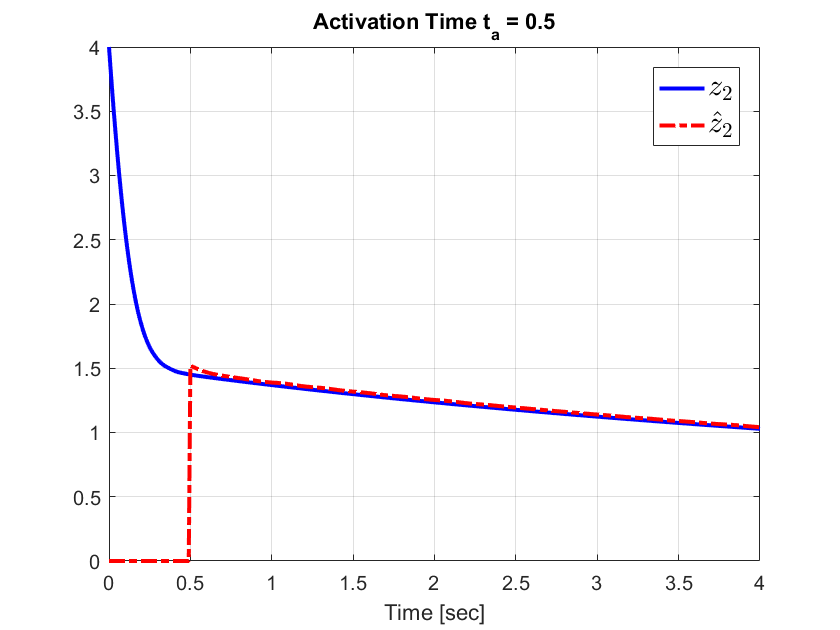}
        \caption{Behaviour of $z_2$ and its estimate $\hat{z}_2$}
    \end{minipage}
\end{figure}

         
\section{Conclusion}\label{sec_con}
In this paper, we applied the observer approach proposed by Djennoune et al. \cite{djennoune2019modulating} to non-linear single input single output systems, with a specific focus on coupled tanks. Our primary contribution lies in the compact reformulation of the equations, which simplifies the analysis of the observer. Furthermore, we demonstrate the effectiveness of this approach through its application to the coupled tank system, showcasing its ability to achieve rapid and accurate state estimation while eliminating the effects of initial conditions.



\bibliography{moad}

\end{document}